\documentclass[11pt]{amsart}

\usepackage[left=1in, right=1in, top=1.2in, bottom=1.2in]{geometry}
\usepackage{microtype, mathtools, mathrsfs, enumerate, dsfont, esvect}
\usepackage{verbatim}
\usepackage[pagebackref, linktocpage]{hyperref}
\hypersetup{
    colorlinks=true,        % false: boxed links; true: colored links
    linkcolor=red,          % color of internal links (change box color with linkbordercolor)
    citecolor=blue,         % color of links to bibliography
    filecolor=magenta,      % color of file links
    urlcolor=cyan           % color of external links
}

\usepackage{amssymb}
\usepackage{url}
\newtheorem{theorem}{Theorem}[section]

\newtheorem{proposition}[theorem]{Proposition}

\newtheorem{lemma}[theorem]{Lemma}

\newtheorem{conjecture}[theorem]{Conjecture}

\theoremstyle{definition}

\newtheorem{case}{Case}

\theoremstyle{plain}
\numberwithin{equation}{theorem}

\theoremstyle{remark}

\newtheorem{remark}[theorem]{Remark}

\newcommand{\Fp}{{\mathbb F}_p}
\newcommand{\Q}{{\mathbb Q}}

\newcommand{\Z}{{\mathbb Z}}
\newcommand{\N}{{\mathbb N}}

\newcommand{\OO}{{\mathcal O}}

\newcommand{\id}{\mathrm{id}}

\newcommand{\Qbar}{\overline{\Q}}

\DeclareMathOperator{\Gal}{Gal}

\newcommand{\bP}{{\mathbb P}}

\newcommand{\bG}{{\mathbb G}}

\newcommand{\bM}{{\mathbb M}}

\newcommand{\bA}{{\mathbb A}}

\newcommand{\lra}{\longrightarrow}

\newcommand{\cH}{\mathscr{H}}

\newcommand{\cP}{\mathscr{P}}

\DeclareMathOperator{\diag}{diag}

\newif\ifhascomments \hascommentstrue
\ifhascomments
  \newcommand{\dragos}[1]{{\color{red}[[\ensuremath{\bigstar\bigstar\bigstar} #1]]}}
  \newcommand{\matt}[1]{{\color{red}[[\ensuremath{\spadesuit\spadesuit\spadesuit} #1]]}}
\else
  \newcommand{\dragos}[1]{}
  \newcommand{\matt}[1]{}
\fi

\newcommand{\bk}{\mathds k}

\newcommand{\arxiv}[1]{\href{https://arxiv.org/abs/#1}{{\tt arXiv:#1}}}

\begin{document}

\title[Density of orbits inside $\bG_a^k\times \bG_m^\ell$]{Density of orbits of endomorphisms of commutative linear algebraic groups}

\author{Dragos Ghioca}
\author{Fei Hu}

\address{Department of Mathematics \\
University of British Columbia\\
Vancouver, BC V6T 1Z2\\
Canada}
\email{dghioca@math.ubc.ca}

\address{ Department of Mathematics \\
University of British Columbia\\
Vancouver, BC V6T 1Z2\\
Canada 
}
\email{fhu@math.ubc.ca}

\begin{abstract}
We prove a conjecture of Medvedev and Scanlon for endomorphisms of connected commutative linear algebraic groups $G$ defined over an algebraically closed field $\bk$ of characteristic $0$. That is, if $\Phi\colon G\lra G$ is a dominant endomorphism, we prove that one of the following holds: either there exists a non-constant rational function $f\in \bk(G)$ preserved by $\Phi$ (i.e., $f\circ \Phi = f$), or there exists a point $x\in G(\bk)$ whose $\Phi$-orbit is Zariski dense in $G$.
\end{abstract}

\subjclass[2010]{
37P15, %Global ground fields
20G15, %Linear algebraic groups over arbitrary fields
32H50. %Iteration problems
}

\keywords{algebraic dynamics, orbit closure, Medvedev--Scanlon conjecture}

\thanks{The first author was partially supported by a Discovery Grant from the National Sciences and Engineering Research Council of Canada. The second author was partially supported by a UBC-PIMS Postdoctoral Fellowship.}

\maketitle

%%%%%%%%%%%%%%%%%%%%%%%%%%%%%%%%%%%%%%%%%%%%%%%%%%%%%%%%%%%%%%%%%%%%
%%%%%%%%%%%%%%%%%%%%%%%%%%%%%%%%%%%%%%%%%%%%%%%%%%%%%%%%%%%%%%%%%%%%

\section{Introduction}
\label{intro section}

\noindent
Throughout our paper, we work over an algebraically closed field $\bk$ of characteristic $0$. Let $\N$ denote the set of positive integers and $\N_0 \coloneqq \N\cup\{0\}$.
For any self-map $\Phi$ on a set $X$, and any $n\in\N_0$, we denote by $\Phi^n$ the $n$-th compositional power, where $\Phi^0$ is the identity map. For any $x\in X$, we denote by $\OO_\Phi(x)$ its forward orbit under $\Phi$, i.e., the set of all iterates $\Phi^n(x)$ for $n\in\N_0$.
An {\it endomorphism} of an algebraic group $G$ is defined as a self-morphism of $G$ in the category of algebraic groups.

Our main result is the following.

\begin{theorem}
\label{main result}
Let $G$ be a connected commutative linear algebraic group defined over an algebraically closed field $\bk$ of characteristic $0$, and $\Phi \colon G\lra G$ a dominant endomorphism. Then either there exists a point $x\in G(\bk)$ such that $\OO_\Phi(x)$ is Zariski dense in $G$, or there exists a non-constant rational function $f\in \bk(G)$ such that $f\circ \Phi = f$.
\end{theorem}

Theorem~\ref{main result} answers affirmatively the following conjecture raised by Medvedev and Scanlon in \cite{medvedev-scanlon} for the case of endomorphisms of $\bG_a^k\times\bG_m^\ell$.
%Here, an {\it endomorphism} of an algebraic group $G$ is defined as a self-morphism of $G$ in the category of algebraic groups.
Note that any connected commutative linear algebraic group splits over an algebraically closed field $\bk$ of characteristic $0$ as a direct product of its largest unipotent subgroup (which is in our case a vector group, i.e., the additive group $\bG_a^k$ of a finite-dimensional $\bk$-vector space) with an algebraic torus $\bG_m^\ell$.

\begin{conjecture}[{cf. \cite[Conjecture~7.14]{medvedev-scanlon}}]
\label{M-S conjecture}
Let $X$ be a quasi-projective variety defined over an algebraically closed field $\bk$ of characteristic $0$, and $\varphi \colon X\dashrightarrow X$ a dominant rational self-map. Then either there exists a point $x\in X(\bk)$ whose orbit under $\varphi$ is Zariski dense in $X$, or $\varphi$ preserves a non-constant rational function $f\in \bk(X)$, i.e., $f \circ \varphi = f$.
\end{conjecture}
With the notation as in Conjecture~\ref{M-S conjecture}, it is immediate to see that if $\varphi$ preserves a non-constant rational function, then there is no Zariski dense orbit. So, the real difficulty in Conjecture~\ref{M-S conjecture} lies in proving that there exists a Zariski dense orbit for a dominant rational self-map $\varphi$ of $X$, which preserves no non-constant rational function.

The origin of \cite[Conjecture~7.14]{medvedev-scanlon} lies in a much older conjecture formulated by Zhang in the early 1990s (and published in \cite[Conjecture~4.1.6]{Zhang-lec}). Zhang asked that for each polarizable endomorphism $\varphi$ of a projective variety $X$ defined over $\Qbar$ there must exist a $\Qbar$-point with Zariski dense orbit under $\varphi$. Medvedev and Scanlon \cite{medvedev-scanlon} conjectured that as long as $\varphi$ does not preserve a non-constant rational function, then a Zariski dense orbit must exist; the hypothesis concerning polarizability of $\varphi$ already implies that no non-constant rational function is preserved by $\varphi$. We describe below the known partial results towards Conjecture~\ref{M-S conjecture}.

\begin{enumerate}
\item In \cite{Amerik-Campana}, Amerik and Campana proved Conjecture~\ref{M-S conjecture} for all uncountable algebraically closed fields $\bk$ (see also \cite{BRS} for a proof of the special case of this result when $\varphi$ is an automorphism). In fact, Conjecture~\ref{M-S conjecture} is true even in positive characteristic, as long as the base field $\bk$ is uncountable (see \cite[Corollary~6.1]{preprint}); on the other hand, when the transcendence degree of $\bk$ over $\Fp$ is smaller than the dimension of $X$, there are counterexamples to the corresponding variant of Conjecture~\ref{M-S conjecture} in characteristic $p$ (as shown in \cite[Example~6.2]{preprint}).
\item In \cite{medvedev-scanlon}, Medvedev and Scanlon proved their conjecture in the special case $X = \bA^n_\bk$ and $\varphi$ is given by the coordinatewise action of $n$ one-variable polynomials $(x_1,\dots, x_n)\mapsto (f_1(x_1),\dots, f_n(x_n))$; their result was established over an arbitrary field $\bk$ of characteristic $0$ which is not necessarily algebraically closed.
\item Conjecture~\ref{M-S conjecture} is known for all projective varieties of positive Kodaira dimension; see for example \cite[Proposition~2.3]{preprint 2}.
\item In \cite{Xie-1}, Conjecture~\ref{M-S conjecture} was proven for all birational automorphisms of surfaces (see also \cite{BGT} for an independent proof of the case of automorphisms). Later, Xie \cite{Xie-2} established the validity of Conjecture~\ref{M-S conjecture} for all polynomial endomorphisms of $\bA^2_\bk$.
\item In \cite{preprint 2}, the conjecture was proven for all smooth minimal $3$-folds of Kodaira dimension $0$ with sufficiently large Picard number, contingent on certain conjectures in the Minimal Model Program.
\item In \cite{GS-abelian}, Conjecture~\ref{M-S conjecture} was proven for all abelian varieties; later this result was extended to dominant regular self-maps for all semi-abelian varieties (see \cite{GS-semiabelian}).
\item In \cite{G-X}, it was proven that if Conjecture~\ref{M-S conjecture} holds for the dynamical system $(X,\varphi)$, then it also holds for the dynamical system $(X \times \bA^k_\bk, \psi)$, where $\psi \colon X \times \bA^k_\bk \dashrightarrow X \times \bA^k_\bk$ is given by $(x, y)\mapsto (\varphi(x), A(x)y)$ and $A\in {\rm GL}_k(\bk(X))$.
\end{enumerate}

We note that combining the results of \cite{GS-semiabelian} (which, in particular, proves Conjecture~\ref{M-S conjecture} when $X = \bG_m^\ell$) with the results of \cite{G-X}, one recovers our Theorem~\ref{main result}. However, our proof of Theorem~\ref{main result} avoids the more complicated arguments from algebraic geometry which were used in the proofs from \cite{G-X} and instead we use mainly number-theoretic tools, employing in a crucial way a theorem of Laurent \cite{Inventiones} regarding polynomial-exponential equations. So, with this new tool which we bring to the study of the Medvedev--Scanlon conjecture, we are able to construct explicitly points with Zariski dense orbits (which is not available in \cite{G-X}). Besides the intrinsic interest in our new approach, as part of our proof, we also obtain in Theorem~\ref{theorem G_a} a more precise result of when a linear transformation has a Zariski dense orbit.

%%%%%%%%%%%%%%%%%%%%%%%%%%%%%%%%%%%%%%%%%%%%%%%%%%%%%%%%%%%%%%%%%%%%%%%%%%%%%%%%
%%%%%%%%%%%%%%%%%%%%%%%%%%%%%%%%%%%%%%%%%%%%%%%%%%%%%%%%%%%%%%%%%%%%%%%%%%%%%%%%

\section{Proof of main results}
\label{section proof} 

We start by proving the following more precise version of the special case in Theorem~\ref{main result} when $G$ is a connected commutative unipotent algebraic group over $\bk$, i.e., $G = \bG_a^k$ for some $k \in \N$.

\begin{theorem}
\label{theorem G_a}
Let $\Phi \colon \bG_a^k\lra \bG_a^k$ be a dominant endomorphism defined over an algebraically closed field $\bk$ of characteristic $0$. Then the following are equivalent:
\begin{enumerate}
\item $\Phi$ preserves a non-constant rational function.
\item There is no $\alpha\in \bG_a^k(\bk)$ whose orbit under $\Phi$ is Zariski dense in $\bG_a^k$. 
\item The matrix $A$ representing the action of $\Phi$ on $\bG_a^k$ is either diagonalizable with multiplicatively dependent eigenvalues, or it has at most $k-2$ multiplicatively independent eigenvalues.
\end{enumerate}
\end{theorem}

\begin{proof}
Clearly, (i) $\Longrightarrow$ (ii). We will prove that (iii) $\Longrightarrow$ (i) and then that (ii) $\Longrightarrow$ (iii). First of all, using \cite[Lemma~5.4]{GS-abelian}, we may assume that $A$ is in Jordan (canonical) form. Strictly speaking, \cite[Lemma~5.4]{GS-abelian} proves that the Medvedev--Scanlon conjecture for abelian varieties is unaffected after replacing the given endomorphism by a conjugate of it through an automorphism; however, its proof goes verbatim for any endomorphism of any quasi-projective variety. Also, because the part~(iii) above is unaffected after replacing $A$ by its Jordan form, then from now on, we assume that $A$ is a Jordan matrix.

Now, assuming (iii) holds, we shall show that (i) holds. Indeed, if $A$ is diagonalizable, then since it has multiplicatively dependent eigenvalues $\lambda_1,\dots, \lambda_k$, i.e., there exist some integers $c_1,\dots, c_k$ not all equal to $0$ such that $\prod_{i=1}^k \lambda_i^{c_i}=1$, then $\Phi$ preserves the non-constant rational function 
$$f \colon \bG_a^k\lra \bP^1_\bk \text{ given by }f(x_1,\dots,x_k) = \prod_{i=1}^k x_i^{c_i},$$
as claimed. Now, assuming $A$ is not diagonalizable and it has at most $k-2$ multiplicatively independent eigenvalues, we will derive (i). There are $3$ easy cases to consider:

\begin{case}
$A$ has $k-2$ Jordan blocks of dimension $1$ and one Jordan block of dimension $2$ and moreover, the corresponding $k-1$ eigenvalues $\lambda_1,\dots, \lambda_{k-1}$ are multiplicatively dependent, i.e., there exist some integers $c_1,\dots, c_{k-1}$ not all equal to $0$ such that $\prod_{i=1}^{k-1} \lambda_i^{c_i}=1$.
Without loss of generality, we may assume that $\lambda_1$ corresponds to the unique Jordan block of dimension $2$. Namely,
\[A = J_{\lambda_1,2} \bigoplus \diag(\lambda_2, \dots, \lambda_{k-1}).\]
Then we conclude that $\Phi$ preserves a non-constant rational function
\[f \colon \bG_a^k\lra \bP^1_\bk \text{ given by }f(x_1,\dots,x_k) = \prod_{i=1}^{k-1} x_{i+1}^{c_i}.\]
\end{case}

\begin{case}
$A$ has at least two Jordan blocks of dimension $2$ each. Again, we may assume that the first two Jordan blocks of $A$ are given by $J_{\lambda_i, 2}$ with $i=1,2$ (it may happen that $\lambda_1 = \lambda_2$).
Then we see that $\Phi$ preserves the non-constant rational function $\bG_a^k\lra \bP^1_\bk$ given by
$$(x_1,\dots,x_k)\mapsto \frac{x_1}{\lambda_2x_2} - \frac{x_3}{\lambda_1x_4}.$$
(Note that $\lambda_1\lambda_2\ne 0$ because the endomorphism $\Phi$ is dominant and hence none of its eigenvalues equals $0$. This is also true in the following cases.) 
\end{case}

\begin{case}
$A$ has a Jordan block of dimension at least equal to $3$ which is denoted by $J_{\lambda, m}$ with $3\le m\le k$.
Clearly, it suffices to prove that the endomorphism $\varphi \colon \bG_a^m\lra\bG_a^m$ (induced by the action of $\Phi$ restricted on the generalized eigenspace corresponding to the eigenvalue $\lambda$) preserves a non-constant rational function. Note that the action of $\varphi$ is given by the Jordan matrix 
\[J_{\lambda,m} = \left(\begin{array}{ccccc}
\lambda & 1 & 0 & \cdots & 0\\
0 & \lambda & 1 & \cdots & 0\\
\vdots & \vdots & \ddots & \ddots & \vdots \\
0 & 0 & \cdots & \lambda & 1\\
0 & 0 & \cdots & 0 & \lambda\\
\end{array}\right).\]
We conclude that $\varphi$ preserves the non-constant rational function $f \colon \bG_a^m\lra \bP^1_\bk$ given by 
$$f(x_1,\dots, x_m) = \frac{2x_{m-2}}{x_m}-\frac{x_{m-1}^2}{x_m^2}+\frac{x_{m-1}}{\lambda x_m}.$$
\end{case}

Therefore, it remains to prove that if (ii) holds, then (iii) must follow. In order to prove this, we show that if $A$ is either diagonalizable with multiplicatively independent eigenvalues, or if $A$ has $k-2$ Jordan blocks of dimension $1$ and one Jordan block of dimension $2$ and moreover, the $k-1$ eigenvalues corresponding to these $k-1$ Jordan blocks are all multiplicatively independent, then there exists a $\bk$-point with a Zariski dense orbit. So, we have two more cases to analyze.

\begin{case}
\label{case:k-evs}
$A$ is diagonalizable with multiplicatively independent eigenvalues $\lambda_1,\dots, \lambda_k$.
In this case, we shall prove that the orbit of $\alpha \coloneqq (1,1,\dots, 1)$ is Zariski dense in $\bG_a^k$. Indeed, if there were a nonzero polynomial $F\in \bk[x_1,\dots, x_k]$ vanishing on the points of the orbit of $\alpha$ under $\Phi$, then we would have that $F(\lambda_1^n,\dots, \lambda_k^n) = F(\Phi^n(\alpha)) = 0$ for each $n\in\N_0$. Let
$$F(x_1,\dots, x_k) = \sum_{(i_1,\dots, i_k)} c_{i_1,\dots, i_k}\prod_{j=1}^k x_j^{i_j},$$
where the coefficients $c_{i_1,\dots, i_k}$'s are nonzero (and clearly, there are only finitely many of them appearing in the above sum). Then it follows that
$$\sum_{(i_1,\dots,i_k)} c_{i_1,\dots, i_k}\cdot \Lambda_{i_1,\dots, i_k}^n=0\text{ for each }n\in\N_0,$$
where $\Lambda_{i_1,\dots, i_k} \coloneqq \prod_{j=1}^k \lambda_j^{i_j}$. On the other hand, since for $(i_1,\dots, i_k)\ne (j_1,\dots, j_k)$ we know that $\Lambda_{i_1,\dots, i_k}/\Lambda_{j_1,\dots,j_k}$ is not a root of unity (because the $\lambda_i$'s are multiplicatively independent), $F(\Phi^n(\alpha))$ is a non-degenerate linear recurrence sequence (see \cite[Definition~3.1]{char p TAMS}). Hence \cite{Schmidt} (see also \cite[Proposition~3.2]{char p TAMS}) yields that as long as $F$ is not identically equal to $0$ (i.e., not all coefficients $c_{i_1,\dots, i_k}$ are equal to $0$), then there are at most finitely many $n\in\N_0$ such that $F(\Phi^n(\alpha))=0$, which is a contradiction. So, indeed, $\OO_\Phi(\alpha)$ is Zariski dense in $\bG_a^k$.
\end{case}

\begin{case}
\label{case:k-1-evs}
$A$ has $k-2$ Jordan blocks of dimension $1$ and one Jordan block of dimension $2$ and moreover, the corresponding $k-1$ eigenvalues $\lambda_1,\dots, \lambda_{k-1}$ are multiplicatively independent.
Without loss of generality, we may assume that
\[A = J_{\lambda_1,2} \bigoplus \diag(\lambda_2, \dots, \lambda_{k-1}) = \left(\begin{array}{ccccc}
\lambda_1 & 1 & 0 & \cdots & 0\\
0 & \lambda_1 & 0 & \cdots & 0\\
0 & 0 & \lambda_2 & \cdots & 0\\
\vdots & \vdots & \vdots & \ddots & \vdots \\
0 & 0 & 0 & \cdots & \lambda_{k-1}
\end{array}\right),\]
and so,
\[A^n = J_{\lambda_1,2}^n \bigoplus \diag(\lambda_2^n, \dots, \lambda_{k-1}^n) = \left(\begin{array}{ccccc}
\lambda_1^n & n\lambda_1^{n-1} & 0 & \cdots & 0\\
0 & \lambda_1^n & 0 & \cdots & 0\\
0 & 0 & \lambda_2^n & \cdots & 0\\
\vdots & \vdots & \vdots & \ddots & \vdots \\
0 & 0 & 0 &\cdots & \lambda_{k-1}^n
\end{array}\right).\]
We shall prove again that the orbit of $\alpha=(1,\dots, 1)$ under the action of $\Phi$ is Zariski dense in $\bG_a^k$. Let $\Psi \colon \bG_a^k\lra \bG_a^k$ be the automorphism given by 
$$(x_1,x_2,x_3,\dots, x_k)\mapsto (\lambda_1(x_1-x_2), x_2,x_3,\dots, x_k)$$
(note that all $\lambda_i$'s are nonzero because $\Phi$ is dominant). It suffices to prove that $\Psi(\OO_\Phi(\alpha))$ is Zariski dense in $\bG_a^k$. This is equivalent with proving that there is no nonzero polynomial $F\in \bk[x_1,\dots, x_k]$ vanishing on 
$$\Psi(\Phi^n(\alpha))=(n\lambda_1^n,\lambda_1^n,\lambda_2^n,\dots, \lambda_{k-1}^n).$$
So, letting $F(x_1,\dots, x_k) \coloneqq \sum_{(i_1,\dots, i_k)}c_{i_1,\dots, i_k}x_1^{i_1}\cdots x_k^{i_k}$, we get that 
\begin{equation}
\label{expanded form F}
F(\Psi(\Phi^n(\alpha))) = \sum_{(i_1,\dots, i_k)} c_{i_1,\dots, i_k}n^{i_1}\left(\lambda_1^{i_1+i_2}\cdot\lambda_2^{i_3}\cdots \lambda_{k-1}^{i_k}\right)^n=0.
\end{equation} 
Letting $\Lambda_{j_1,\dots, j_{k-1}} \coloneqq \lambda_1^{j_1}\cdot \lambda_2^{j_2}\cdots \lambda_{k-1}^{j_{k-1}}$, we can rewrite \eqref{expanded form F} as
\begin{equation}
\label{expanded form F 2}
\sum_{(j_1,\dots, j_{k-1})} Q_{j_1,\dots, j_{k-1}}(n)\cdot \Lambda_{j_1,\dots, j_{k-1}}^n=0,
\end{equation}
where 
$$Q_{j_1,\dots, j_{k-1}}(n) \coloneqq \sum_{\substack{i_1+i_2=j_1 \, \text{and} \\ i_3=j_2,\dots, i_k=j_{k-1}}}c_{i_1,i_2,i_3,\dots,i_k}n^{i_1}.$$
As in the previous Case \ref{case:k-evs}, the left-hand side of \eqref{expanded form F 2} represents the general term of a non-degenerate linear recurrence sequence (i.e., such that the quotient of any two of its distinct characteristic roots is not a root of unity). It follows from \cite{Schmidt} (see also \cite[Proposition~3.2]{char p TAMS}) that there are at most finitely many $n\in\N_0$ such that \eqref{expanded form F 2} holds, unless $F=0$ (i.e., each coefficient $c_{i_1,\dots, i_k}$ equals $0$). Therefore, $\Psi(\OO_\Phi(\alpha))$ is indeed Zariski dense in $\bG_a^k$ and hence so is $\OO_\Phi(\alpha)$.
\end{case}

This concludes our proof of Theorem~\ref{theorem G_a}.
\end{proof}

\begin{remark}
\label{rem:more precise}
We note that in Theorem~\ref{theorem G_a} we actually proved a stronger statement as follows. If $A$ is a Jordan matrix acting on $\bG_a^k$ and either it has $k$ multiplicatively independent eigenvalues, or it is not diagonalizable, but it still has $k-1$ multiplicatively independent eigenvalues, then there is no proper subvariety of $\bG_a^k$ which contains infinitely many points from the orbit of $(1,\dots, 1)$ under the action of $A$. So, not only that the orbit of $(1,\dots, 1)$ is Zariski dense in $\bG_a^k$, but any infinite subset of its orbit must also be Zariski dense in $\bG_a^k$. This strengthening is similar to the one obtained in \cite[Corollary~1.4]{AJM} for the action of any \'etale endomorphism of a quasi-projective variety (see also \cite{the book} for the connections of these results to the dynamical Mordell-Lang conjecture).
\end{remark}

The next result will be used in our proof of Theorem~\ref{main result}.
\begin{proposition}
\label{prop:bounded}
Let $A\in \bM_{\ell,\ell}(\Z)$ be a matrix with nonzero determinant, and let $\vv{p} \in \bM_{\ell,1}(\Z)$ be a nonzero vector. Let $c_1$ and $c_2$ be positive real numbers. If there exists an infinite set $S$ of positive integers such that for each $n\in S$, we have that $A^n\cdot \vv{p}$ is a vector whose entries are all bounded in absolute value by $c_1n+c_2$, then $A$ has an eigenvalue which is a root of unity.
\end{proposition}

\begin{proof}
Let $B\in \bM_{\ell,\ell}(\Qbar)$ be an invertible matrix such that $J\coloneqq BAB^{-1}$ is the Jordan canonical form of $A$. For each $n\in \N$, let $\vv{p_n}\coloneqq A^n\cdot \vv{p}$ and $\vv{q_n}\coloneqq B\cdot \vv{p_n}$. So, we know that each entry in $\vv{p_n}$ is an integer bounded in absolute value by $c_1n+c_2$ for any $n \in S \subseteq \N$.
Then, according to our hypotheses, there exist some positive constants $c_3$ and $c_4$ such that each entry in $\vv{q_n}$ is bounded in absolute value by $c_3n+c_4$. Furthermore, for any $\sigma\in\Gal(\Qbar/\Q)$, denoting by $\vv{v}^\sigma$ the vector obtained by applying $\sigma$ to each entry of the vector $\vv{v}\in \bM_{\ell,1}(\Qbar)$, we have that each entry in $\vv{q_n}^\sigma$ is bounded by $c_5n+c_6$ for some positive constants $c_5$ and $c_6$ which are independent of $n$ and $\sigma$. Indeed, this claim follows from the observation that $\vv{q_n}^\sigma=B^\sigma\cdot \vv{p_n}$, because $\vv{p_n}$ has integer entries (since both $\vv{p}$ and $A$ have integer entries) and moreover, the entries in $\vv{p_n}$ are all bounded in absolute value by $c_1n+c_2$.

Denote by $\ell_1, \dots, \ell_m$ the dimensions of the Jordan blocks of $J$ in the order as they appear in the matrix $J$ (so, $\ell = \ell_1+\cdots +\ell_m$). Let $\vv{q}\coloneqq B\cdot \vv{p}$. Since $\vv{p}\ne \vv{0}$ and $B$ is invertible, then $\vv{q}$ is not the zero vector either. Without loss of generality, we may assume that one of the first $\ell_1$ entries in $\vv{q}$ is nonzero. Next, we will prove that the eigenvalue of $J$ corresponding to its first Jordan block (of dimension $\ell_1$) must have absolute value at most equal to $1$. We state and prove our result from Lemma~\ref{lem:bounded} in much higher generality than needed since it holds for any valued field $(L,|\cdot |)$ (our application will be for $L=\Qbar$ equipped with the usual archimedean absolute value $|\cdot |$).

\begin{lemma}
\label{lem:bounded}
Let $(L,|\cdot |)$ be an arbitrary valued field, let $J_{\lambda_1, r}\in \bM_{r, r}(L)$ be a Jordan block of dimension $r\ge 1$ corresponding to a nonzero eigenvalue $\lambda_1$, and let $\vv{v}\in \bM_{r, 1}(L)$ be a nonzero vector. If there exist positive constants $c_5$, $c_6$, and an infinite set $S_1$ of positive integers such that for each $n\in S_1$, we have that each entry in $J_{\lambda_1, r}^n\cdot \vv{v}$ is bounded in absolute value by $c_5n+c_6$, then $|\lambda_1|\le 1$. 
\end{lemma}

\begin{proof}[Proof of Lemma~\ref{lem:bounded}] \renewcommand{\qedsymbol}{}
Let $s$ be the largest integer with the property that the $s$-th entry $v_s$ in $\vv{v}$ is nonzero; so, $1\le s\le r$. Then for each $n\in S_1$, we have that the $s$-th entry in $J_{\lambda_1, r}^n\cdot \vv{v}$ is $v_s\lambda_1^n$ and hence according to our hypothesis, we have
\begin{equation}
\label{bounded lambda 1}
\left|v_s\lambda_1^n\right| \le c_5n+c_6.
\end{equation}
Since $v_s\ne 0$ and equation \eqref{bounded lambda 1} holds for each $n$ in the infinite set $S_1$, we conclude that $|\lambda_1|\le 1$, as desired. Thus, the lemma follows.
\end{proof}

So, our assumptions coupled with Lemma~\ref{lem:bounded} yield that the eigenvalue $\lambda_1$ corresponding to the first Jordan block of the matrix $J$ has absolute value at most equal to $1$. Furthermore, as previously explained, for each $n\in S$ and for each $\sigma\in\Gal(\Qbar/\Q)$, we have that each entry in
$$\vv{q_n}^\sigma = (B\cdot \vv{p_n})^\sigma = \left(BA^n\cdot \vv{p}\right)^\sigma = \left(J^n\cdot \vv{q}\right)^\sigma = \left(J^\sigma\right)^n\cdot \vv{q}^\sigma$$
is bounded in absolute value by $c_5n+c_6$. Thus, applying again Lemma~\ref{lem:bounded}, this time to the first Jordan block of the matrix $J^\sigma$, we obtain that $|\sigma(\lambda_1)|\le 1$.

Now, $\lambda_1$ is an algebraic integer (since it is the eigenvalue of a matrix with integer entries) and for each $\sigma\in\Gal(\Qbar/\Q)$, we have that $|\sigma(\lambda_1)|\le 1$. Because the product of all the Galois conjugates of $\lambda_1$ must be a nonzero integer, we conclude that actually $|\sigma(\lambda_1)|=1$ for each $\sigma\in\Gal(\Qbar/\Q)$. Then a classical lemma from algebraic number theory yields that $\lambda_1$ must be a root of unity, as desired.
\end{proof}

Now we are ready to prove our main theorem stated in the introduction.

\begin{proof}[Proof of Theorem~\ref{main result}.]
Because $G$ is a connected commutative linear algebraic group defined over an algebraically closed field $\bk$ of characteristic $0$, then $G$ is isomorphic to $\bG_a^k\times \bG_m^\ell$ for some $k,\ell\in\N_0$. Since there are no nontrivial homomorphisms between $\bG_a$ and $\bG_m$, then $\Phi$ splits as $\Phi_1\times \Phi_2$, where $\Phi_1$ and $\Phi_2$ are dominant endomorphisms of $\bG_a^k$ and $\bG_m^\ell$, respectively. So, our conclusion follows once we prove the following statement: if neither $\Phi_1$ nor $\Phi_2$ preserve any non-constant rational function, then there exists a point $\alpha\in (\bG_a^k\times \bG_m^\ell)(\bk)$ with a Zariski dense orbit under $\Phi$.

Thus, we assume that $\Phi_1$ and $\Phi_2$ do not preserve any non-constant rational function. In particular, this means that the action of $\Phi_2$ on the tangent space of the identity of $\bG_m^\ell$ is given through a matrix $A_2$ whose eigenvalues are not roots of unity (since otherwise one may argue as in the proof of \cite[Lemma~6.2]{GS-abelian} or \cite[Lemma~4.1]{GS-semiabelian} that $\Phi_2$ preserves a non-constant fibration which is not the case). Also, our Theorem~\ref{theorem G_a} yields that either the matrix $A_1$ (which represents $\Phi_1$) is diagonalizable with multiplicatively independent eigenvalues, or the Jordan canonical form of $A_1$ has $k-2$ blocks of dimension $1$ and one block of dimension $2$ such that the $k-1$ eigenvalues are multiplicatively independent. Next, we will analyze in detail the second possibility for $A_1$ (when there is a Jordan block of dimension $2$), since the former possibility (when $A_1$ is diagonalizable with multiplicatively independent eigenvalues) turns out to be a special case of the latter one.

Arguing as in the proof of Theorem~\ref{theorem G_a}, at the expense of replacing $\Phi_1$ and therefore $\Phi$ by a conjugate through an automorphism, we may assume that $A_1$ is a Jordan matrix of the form
\[A_1 = J_{\lambda_1,2} \bigoplus \diag(\lambda_2, \dots, \lambda_{k-1}) = \left(\begin{array}{ccccc}
\lambda_1 & 1 & 0 & \cdots & 0\\
0 & \lambda_1 & 0 & \cdots & 0\\
0 & 0 & \lambda_2 & \cdots & 0\\
\vdots & \vdots & \vdots & \ddots & \vdots \\
0 & 0 & 0 & \cdots & \lambda_{k-1}
\end{array}\right),\]
where $\lambda_1,\dots,\lambda_{k-1}$ are multiplicatively independent eigenvalues. We will prove that there exists a point $\alpha\in (\bG_a^k\times\bG_m^\ell)(\bk)$ with a Zariski dense orbit. Suppose that we have proved it for the time being. Then restricting the action of $\Phi_1$ (and thus of $A_1$) to the last $k-1$ coordinate axes of $\bG_a^k$, we obtain a diagonal matrix with multiplicatively independent eigenvalues. Letting $\hat\pi_1$ be the projection of $\bG_a^k$ to $\bG_a^{k-1}$ with the first coordinate omitted, we obtain a point $\gamma \coloneqq (\hat\pi_1 \times \id_{\bG_m^\ell})(\alpha)$ whose orbit under the induced endomorphism of $\bG_a^{k-1}\times \bG_m^\ell$ is Zariski dense. This justifies our earlier claim that it suffices to consider the case of a non-diagonalizable linear action $\Phi_1$ since the diagonal case reduces to this more general case.

Let $\alpha\coloneqq (\alpha_1,\dots, \alpha_k,\beta_{1},\dots, \beta_{\ell})\in (\bG_a^k\times\bG_m^\ell)(\Qbar)$ such that $\alpha_1 = \cdots = \alpha_k=1$, while $\lambda_1,\dots, \lambda_{k-1}$, $\beta_{1},\dots, \beta_{\ell}$ are all multiplicatively independent. We will prove that $\OO_\Phi(\alpha)$ is Zariski dense in $\bG_a^k\times \bG_m^\ell$. Since $\lambda_1,\dots, \lambda_{k-1}$ are multiplicatively independent elements of $\bk$ (which is an algebraically closed field containing $\Qbar$), without loss of generality, we may assume that each $\lambda_i\in\Qbar$. This follows through a standard specialization argument as shown in \cite[Section~5]{Masser} (see also \cite[p.~39]{Zannier}); one can actually prove that there are infinitely many specializations which would yield multiplicatively independent $\lambda_1,\dots, \lambda_{k-1},\beta_1,\dots, \beta_\ell$. (Note that if the orbit of $\alpha$ under the action of a specialization of $\Phi$ has a Zariski dense orbit, then $\OO_\Phi(\alpha)$ must itself be Zariski dense in $\bG_a^k\times \bG_m^\ell$.)

Now, suppose to the contrary that there is a hypersurface $Y$ (not necessarily irreducible) of $\bG_a^k\times \bG_m^\ell$ containing $\OO_\Phi(\alpha)$. Similar to the proof of Theorem~\ref{theorem G_a} (see the Case~\ref{case:k-1-evs}), considering the birational automorphism $\Psi_1 \colon \bG_a^k\dashrightarrow \bG_a^k$ given by
$$(x_1,x_2,x_3,\dots,x_k)\mapsto \left(\frac{\lambda_1(x_1-x_2)}{x_2}, x_2,x_3,\dots, x_k\right),$$
which extends to a birational automorphism $\Psi\coloneqq \Psi_1\times \id_{\bG_m^\ell}$ of $\bG_a^k\times \bG_m^\ell$, we see that $\Psi(Y)$ is a hypersurface of $\bG_a^k\times \bG_m^\ell$ containing $\Psi(\OO_\Phi(\alpha))$. In particular, this yields that there exists some nonzero polynomial $F\in \Qbar[x_1,\dots, x_{k+\ell}]$ (since the entire orbit of $\alpha$ is defined over $\Qbar$) vanishing at the following set of $\Qbar$-points:
\begin{equation}
\label{general form point}
\Psi(\OO_\Phi(\alpha)) = \left\{(n,\lambda_1^n,\lambda_2^n,\dots, \lambda_{k-1}^n,\beta_{n,1},\beta_{n,2},\dots, \beta_{n,\ell}) \in (\bG_a^k \times \bG_m^\ell)(\Qbar) : n\in \N_0 \right\},
\end{equation}
where $(\beta_{n,1},\dots, \beta_{n,\ell})\coloneqq \Phi_2^n(\beta_1,\dots, \beta_\ell)$. So, letting $\left\{m^{(n)}_{i,j}\right\}_{1\le i,j\le \ell}$ be the entries of the matrix $A_2^n$, then the point $\Phi_2^n(\beta_1,\dots, \beta_\ell)\in \bG_m^\ell(\Qbar)$ equals
$$\left(\prod_{j=1}^\ell \beta_j^{m^{(n)}_{1,j}}, \dots , \prod_{j=1}^\ell \beta_j^{m^{(n)}_{\ell,j}}\right),$$
or alternatively, we can write it as $\beta^{A_2^n}$, where $\beta\coloneqq (\beta_1,\dots, \beta_\ell)\in \bG_m^\ell(\Qbar)$. More generally, for a matrix $M\in \bM_{\ell,\ell}(\Z)$ and some $\gamma\coloneqq (\gamma_1,\dots,\gamma_\ell)\in \bG_m^\ell(\Qbar)$, we let $\gamma^M$ be $\varphi(\gamma)$, where $\varphi \colon \bG_m^\ell\lra \bG_m^\ell$ is the endomorphism corresponding to the matrix $M$ (with respect to the action of $\varphi$ on the tangent space of the identity of $\bG_m^\ell$). Furthermore, for any $\vv{a}\coloneqq (a_1,\dots, a_\ell)\in\Z^\ell$, we let $\gamma^{\vv{a}}\in\bG_m(\Qbar)$ be $\prod_{i=1}^\ell \gamma_i^{a_i}$. 

We also write 
$$F(x_1,\dots, x_{k+\ell}) = \sum_{(i_1,\dots, i_{k+\ell})}c_{i_1,\dots, i_{k+\ell}}x_1^{i_1}\cdots x_{k+\ell}^{i_{k+\ell}},$$
where each coefficient $c_{i_1,\dots, i_{k+\ell}}$ is nonzero so that it is a finite sum. We denote $\vv{i_{2,\dots, k}}\coloneqq (i_2,\dots,i_k)\in\Z^{k-1}$, $\vv{i_{k+1,\dots, k+\ell}}\coloneqq (i_{k+1},\dots, i_{k+\ell})\in \Z^{\ell}$, and $\Lambda\coloneqq (\lambda_1,\dots, \lambda_{k-1})\in \bG_m^{k-1}(\Qbar)$. Note that the $\lambda_i$'s are nonzero since $\Phi_1$ is a dominant endomorphism. Let $M\coloneqq (m_{r,s})\in \bM_{\ell,\ell}(\Z)$ be a matrix of integer variables and consider the polynomial-exponential equation
\begin{equation}
\label{the equation}
\sum_{(i_2,\dots, i_{k+\ell})}\left(\sum_{i_1}c_{i_1,\dots, i_{k+\ell}}n^{i_1}\right)\cdot \left( \Lambda^{\vv{i_{2,\dots, k}}} \right)^n \cdot \beta^{\,\vv{i_{k+1,\dots, k+\ell}}\,\cdot M}=0;
\end{equation}
in particular, $\beta^{\,\vv{i_{k+1,\dots, k+\ell}}\,\cdot M}$ equals
$$\prod_{s=1}^\ell \beta_s^{\sum_{r=1}^\ell i_{k+r}m_{r,s}} = \prod_{r,s=1}^\ell \beta_s^{i_{k+r}m_{r,s}}.$$
With the notation as in \eqref{the equation}, we let 
$$Q_{\vv{i_{2,\dots, k+\ell}}}(n) \coloneqq \sum_{i_1}c_{i_1,\dots, i_{k+\ell}}n^{i_1}.$$ 
So, the polynomial-exponential equation \eqref{the equation} has $\ell^2+1$ integer variables; denoting $\Lambda_{i_2,\dots, i_k} \coloneqq \Lambda^{\vv{i_{2,\dots, k}}}$, we have 
\begin{equation}
\label{the equation 2}
\sum_{(i_2,\dots, i_{k+\ell})} Q_{\vv{i_{2,\dots, k+\ell}}}(n) \cdot \Lambda_{i_2,\dots, i_k}^n \cdot \prod_{r,s=1}^\ell \left(\beta_s^{i_{k+r}}\right)^{m_{r,s}}=0.
\end{equation}

We are going to apply \cite[Th\'eor\`eme~6]{Inventiones}. Note that each $n\in\N_0$ for which $$F(\Psi(\Phi^n(\alpha)))=0$$ yields an integer solution $\left(n,\left(m_{i,j}^{(n)}\right)_{1\le i,j\le \ell}\right)$ of the equation \eqref{the equation 2}.
Now, for each $n\in\N_0$, we let $\cP_n$ be a maximal compatible partition of the set of indices $(i_2,\dots, i_{k+\ell})$ in the sense of Laurent (see \cite[p.~320]{Inventiones}) with the property that for each part $I$ of the partition $\cP_n$, we have that
\begin{equation}
\label{be in I}
\sum_{(i_2,\dots, i_{k+\ell})\in I} Q_{\vv{i_{2,\dots, k+\ell}}}(n) \cdot \Lambda_{i_2,\dots, i_k}^n \cdot \prod_{r,s=1}^\ell \left(\beta_s^{i_{k+r}}\right)^{m^{(n)}_{r,s}}=0.
\end{equation}
Since there are only finitely many partitions of the finite index set of all $(i_2,\dots, i_{k+\ell})$, we fix some partition $\cP$ for which we assume that there exists an infinite set $S$ of positive integers $n$ such that $\cP\coloneqq \cP_n$. Then we define $\cH_\cP$ as the subgroup of $\Z^{\ell^2+1}$ consisting of all $\left(n, \left(m_{i,j}^{(n)}\right)_{1\le i,j\le \ell}\right)$ such that for each part $I$ of the partition $\cP$ and for any two indices $\vv{i}\coloneqq (i_2,\dots, i_{k+\ell})$ and $\vv{j}\coloneqq (j_2,\dots, j_{k+\ell})$ contained in $I$, we have that
\begin{equation}
\label{be in H}
\Lambda_{i_2,\dots, i_k}^n \cdot \prod_{r,s=1}^\ell \left(\beta_s^{i_{k+r}}\right)^{m^{(n)}_{r,s}} = \Lambda_{j_2,\dots, j_k}^n \cdot \prod_{r,s=1}^\ell \left(\beta_s^{j_{k+r}}\right)^{m^{(n)}_{r,s}}.
\end{equation}
Then \cite[Th\'{e}or\`{e}me~6]{Inventiones} yields that we can write the solution $\left(n, \left(m^{(n)}_{i,j}\right)_{1\le i, j\le \ell}\right)$ as $\vv{N_0}(n)+\vv{N_1}(n)$, where $\vv{N_0} \coloneqq \vv{N_0}(n), \, \vv{N_1} \coloneqq \vv{N_1}(n)\in\Z^{1+\ell^2}$ and moreover, $\vv{N_0}\in \cH_\cP$ while the absolute value of each entry in $\vv{N_1}$ is bounded above by $C_1\log(U_n)+C_2$, where $C_1$ and $C_2$ are some positive constants independent of $n$, and
$$U_n\coloneqq \max\left\{n, \max_{1\le i,j\le \ell} \big|m^{(n)}_{i,j}\big|\right\}.$$
A simple computation for $A_2^n = \left(m^{(n)}_{i,j}\right)_{1\le i,j\le \ell}$ yields that there exists a positive constant $C_3$ such that $U_n\le C_3^n$ for all $n\in\N$. We then conclude that each entry in $\vv{N_1}$ is bounded in absolute value by $C_4n+C_5$, for some absolute constants $C_4$ and $C_5$ independent of $n$. Next, we will determine the subgroup $\cH_\cP$ of $\Z^{1+\ell^2}$.

We may first assume that at least one part $I$ of the partition $\cP$ satisfies the property that $\pi(I)$ has at least $2$ elements, where $\pi \colon \Z^{k+\ell-1}\lra \Z^\ell$ is the projection on the last $\ell$ coordinates, i.e., $(i_2, \dots, i_{k+\ell}) \mapsto (i_{k+1}, \dots, i_{k+\ell})$.
Indeed, if $\#(\pi(I))=1$ for each part $I$ of $\cP$, then equation \eqref{be in I} would actually yield that 
\begin{equation}
\label{be in I 2}
\sum_{(i_2,\dots, i_{k+\ell})\in I} Q_{\vv{i_{2,\dots, k+\ell}}}(n) \cdot \Lambda_{i_2,\dots, i_k}^n = 0.
\end{equation}
Thus, since \eqref{be in I 2} holds for each part $I$ of $\cP$, we would get that there exists a proper subvariety of $\bG_a^k\times \bG_m^\ell$ of the form $Z\times \bG_m^\ell$ containing infinitely many points of $\Psi(\OO_\Phi(\alpha))$. In particular, $Z$ would be a proper subvariety of $\bG_a^k$ containing infinitely many points of $\Psi_1(\OO_{\Phi_1}(1,\dots,1))$, which contradicts the proof of Theorem~\ref{theorem G_a} (see also Remark~\ref{rem:more precise}). Therefore, we may indeed assume that there exists at least one part $I$ of $\cP$ such that $\pi(I)$ contains at least two distinct elements $(i_{k+1},\dots, i_{k+\ell})$ and $(j_{k+1},\dots, j_{k+\ell})$. 

Let $\vv{N_0}\coloneqq \left(n_0,\left(m^{(n)}_{0,i,j}\right)_{1\le i,j\le \ell}\right)$.
Since $\vv{N_0} \in \cH_\cP$, by the definition of $\cH_\cP$,
we apply \eqref{be in H} to $\vv{N_0}$ and to $(i_2,\dots, i_{k+\ell}), \, (j_2,\dots, j_{k+\ell})\in I$ for which $(i_{k+1},\dots, i_{k+\ell})\ne (j_{k+1},\dots, j_{k+\ell})$ and get that 
\begin{equation}
\label{be in H 2}
\Lambda_{i_2,\dots, i_k}^{n_0} \cdot \prod_{r,s=1}^\ell \left(\beta_s^{i_{k+r}}\right)^{m^{(n)}_{0,r,s}} = \Lambda_{j_2,\dots, j_k}^{n_0} \cdot \prod_{r,s=1}^\ell \left(\beta_s^{j_{k+r}}\right)^{m^{(n)}_{0,r,s}}.
\end{equation}
Using the fact that $\Lambda_{i_2,\dots, i_k} = \prod_{t=1}^{k-1}\lambda_t^{i_{t+1}}$ and that $\lambda_1,\dots, \lambda_{k-1},\beta_1,\dots, \beta_{\ell}$ are multiplicatively independent, equation \eqref{be in H 2} yields that 
\begin{equation}
\label{be in H 3}
\sum_{r=1}^\ell i_{k+r}m^{(n)}_{0,r,s} = \sum_{r=1}^\ell j_{k+r}m^{(n)}_{0,r,s} \text{ for any } 1\le s\le \ell.
\end{equation}
Denote $M_n^0\coloneqq \left(m^{(n)}_{0,r,s}\right)_{1\le r,s\le \ell}$ and also let $\vv{p}\coloneqq (i_{k+1}-j_{k+1},\dots, i_{k+\ell}-j_{k+\ell})^t \in \bM_{\ell,1}(\Z)$. Then we may write equation \eqref{be in H 3} as $\vv{p}^t\cdot M_n^0 = \vv{0}$.

Let $\vv{N_1}\coloneqq \left(n_1,\left(m^{(n)}_{1,r,s}\right)_{1\le r,s\le \ell}\right)$ and denote $M_n^1\coloneqq \left(m^{(n)}_{1,r,s}\right)_{1\le r,s\le \ell}$. Then we have $A_2^n=M_n^0+M_n^1$ for each $n\in S$, i.e., $m^{(n)}_{r,s}=m^{(n)}_{0,r,s}+m^{(n)}_{1,r,s}$ for each $1\le r,s\le \ell$. Using that $\vv{p}^t\cdot M_n^0 = \vv{0}$, we obtain that for each $n\in S$ we have 
\begin{equation}
\label{impossible equation 1}
\vv{p}^t\cdot A_2^n = \vv{p}^t\cdot M_n^1, \, \text{or equivalently, }(A_2^t)^n\cdot \vv{p}=(M_n^1)^t\cdot\vv{p},
\end{equation}
where $D^t$ always represents the transpose of the matrix $D$. 
Using the fact that each entry in $(M_n^1)^t$ is bounded in absolute value by $C_4n+C_5$, we obtain that each entry of the vector 
\begin{equation}
\label{vv equation}
\vv{p_n}\coloneqq (A_2^t)^n \cdot \vv{p} = (M_n^1)^t\cdot \vv{p} 
\end{equation}
is also bounded in absolute value by $C_6n+C_7$ (again for some positive constants $C_6$ and $C_7$ independent of $n$). Note that $\vv{p}\ne\vv{0}$ and \eqref{vv equation} holds for all $n$ in the infinite set $S$ of positive integers. It follows from Proposition~\ref{prop:bounded} that one of the eigenvalues of $A_2$ must be a root of unity, which contradicts our assumption on $A_2$ at the beginning of the proof.

This concludes our proof of Theorem~\ref{main result}.
\end{proof}

%%%%%%%%%%%%%%%%%%%%%%%%%%%%%%%%%%%%%%%%%%%%%%%%%%%%%%%%%%%%%%%%%%%%%%%%%%%%%%%
%%%%%%%%%%%%%%%%%%%%%%%%%%%%%%%%%%%%%%%%%%%%%%%%%%%%%%%%%%%%%%%%%%%%%%%%%%%%%%%

\phantomsection
\addcontentsline{toc}{section}{Acknowledgments}
\noindent
\textbf{Acknowledgments. }
We thank Zinovy Reichstein, Tom Scanlon and Umberto Zannier for useful discussions regarding this topic.
We are also grateful to the referee for helpful comments.

%%%%%%%%%%%%%%%%%%%%%%%%%%%%%%%%%%%%%%%%%%%%%%%%%%%%%%%%%%%%%%%%%%%%%%%%%%%%%%
%%%%%%%%%%%%%%%%%%%%%%%%%%%%%%%%%%%%%%%%%%%%%%%%%%%%%%%%%%%%%%%%%%%%%%%%%%%%%%


\begin{thebibliography}{BGRS10}

\bibitem[AC08]{Amerik-Campana}
E.~Amerik and F.~Campana, \emph{Fibrations m\'eromorphes sur certaines vari\'et\'es \`a fibr\'e canonique trivial}, Pure Appl. Math. Q. \textbf{4} (2008), no.~2, 509--545.

\bibitem[BGR17]{preprint}
J.~P.~Bell, D.~Ghioca, and Z.~Reichstein, \emph{On a dynamical version of a theorem of Rosenlicht}, Ann. Sc. Norm. Super. Pisa Cl. Sci. (5) \textbf{17} (2017), no.~1, 187--204.

\bibitem[BGRS17]{preprint 2}
J.~P.~Bell, D.~Ghioca, Z.~Reichstein, and M.~Satriano, \emph{On the Medvedev--Scanlon conjecture for minimal threefolds of non-negative Kodaira dimension}, New York J. Math. \textbf{23} (2017), 1185--1203.

\bibitem[BGT10]{AJM}
J.~P.~Bell, D.~Ghioca, and T.~J.~Tucker, \emph{The dynamical Mordell-Lang problem for \'etale maps}, Amer. J. Math. \textbf{132} (2010), no.~6, 1655--1675.

\bibitem[BGT15]{BGT}
\bysame, \emph{Applications of $p$-adic analysis for bounding periods of subvarieties under \'etale maps}, Int. Math. Res. Not. IMRN \textbf{2015}, no.~11, 3576--3597.

\bibitem[BGT16]{the book}
\bysame, \emph{The dynamical Mordell-Lang conjecture}, Mathematical Surveys and Monographs, vol. 210, American Mathematical Society, Providence, RI, 2016.

\bibitem[BRS10]{BRS} 
J.~P.~Bell, D.~Rogalski, and S.~J.~Sierra, \emph{The Dixmier-Moeglin equivalence for twisted homogeneous coordinate rings}, Israel J. Math. \textbf{180} (2010), 461--507.

\bibitem[Ghi]{char p TAMS}
D.~Ghioca, \emph{The dynamical Mordell-Lang conjecture in positive characteristic}, Trans. Amer. Math. Soc., to appear, \arxiv{1610.00367}, 18 pp.

\bibitem[GS]{GS-semiabelian}
D.~Ghioca and M.~Satriano, \emph{Density of orbits of dominant regular self-maps of semiabelian varieties}, Trans. Amer. Math. Soc., to appear, \arxiv{1708.06221}, 17 pp.

\bibitem[GS17]{GS-abelian}
D.~Ghioca and T.~Scanlon, \emph{Density of orbits of endomorphisms of abelian varieties}, Trans. Amer. Math. Soc. \textbf{369} (2017), no.~1, 447--466.

\bibitem[GX]{G-X}
D.~Ghioca and J.~Xie, \emph{Algebraic dynamics of skew-linear self-maps}, Proc. Amer. Math. Soc., to appear, retrieved from \url{https://www.math.ubc.ca/~dghioca/papers/revision_PAMS.pdf}, 16 pp.

\bibitem[Lau84]{Inventiones}
M.~Laurent, \emph{\'{E}quations diophantiennes exponentielles}, Invent. Math. \textbf{78} (1984), no.~2, 299--327.

\bibitem[Mas89]{Masser}
D.~Masser, \emph{Specializations of finitely generated subgroups of abelian varieties}, Trans. Amer. Math. Soc. \textbf{311} (1989), no.~1, 413--424.

\bibitem[MS14]{medvedev-scanlon}
A. Medvedev and T. Scanlon, \emph{Invariant varieties for polynomial dynamical systems}, Ann. of Math. (2) \textbf{179} (2014), no. 1, 81--177.

\bibitem[Sch03]{Schmidt}
W.~M.~Schmidt, \emph{Linear recurrence sequences}, Diophantine approximation (Cetraro, 2000), Lecture Notes in Math., vol.~1819, Springer, Berlin, 2003, pp.~171--247.

\bibitem[Xie15]{Xie-1}
J.~Xie, \emph{Periodic points of birational transformations on projective surfaces}, Duke Math. J. \textbf{164} (2015), no.~5, 903--932. 

\bibitem[Xie17]{Xie-2}
\bysame, \emph{The existence of Zariski dense orbits for polynomial endomorphisms of the affine plane}, Compos. Math. \textbf{153} (2017), no.~8, 1658--1672.

\bibitem[Zan12]{Zannier}
U.~Zannier, \emph{Some problems of unlikely intersections in arithmetic and geometry}, Annals of Mathematics Studies, vol.~181, Princeton University Press, Princeton, NJ, 2012. With appendixes by David Masser.

\bibitem[Zha10]{Zhang-lec}
S.-W. Zhang, \emph{Distributions in algebraic dynamics}, Surveys in differential geometry. Vol. X, Surv. Differ. Geom., vol. 10, Int. Press, Somerville, MA, 2006, pp. 381--430.

\end{thebibliography}
\end{document}